\documentclass[12pt,reqno]{amsart}
\usepackage{latexsym}
\usepackage{amsmath,amssymb,amsfonts}
\usepackage{mathrsfs}
\usepackage[abbrev]{amsrefs}
\usepackage[pdftex]{graphicx}

\usepackage{latexsym}
\usepackage{delarray} 
\usepackage{empheq}

\usepackage{enumerate}
\usepackage{mathtools}

\newtheorem{thm}{Theorem}[section]
\newtheorem{cor}[thm]{Corollary}
\newtheorem{prop}[thm]{Proposition}
\newtheorem{lem}[thm]{Lemma}
\newtheorem{rem}[thm]{Remark}

\newtheorem{definition}[thm]{Definition}

\newcommand{\dist}{\mathop{\text{dist}}}

\DeclareMathOperator*{\supp}{supp}

\numberwithin{equation}{section}
\numberwithin{thm}{section}

\makeatletter
\newcommand{\vertiii}[1]{{\left\vert\kern-0.25ex\left\vert\kern-0.25ex\left\vert #1 
    \right\vert\kern-0.25ex\right\vert\kern-0.25ex\right\vert}}

\title[Boundedness of composition operators on Besov spaces]{Boundedness of composition operators on higher order Besov spaces in one dimension}
\author[M. Ikeda, I. Ishikawa and K. Taniguchi]
{Masahiro Ikeda, Isao Ishikawa and Koichi Taniguchi}
\address[M. Ikeda]
{Center for Advanced Intelligence Project, RIKEN
1-4-1, Nihonbashi, Chuo-ku, Tokyo 103-0027, Japan/
Faculty of Science and Technology,
Keio University, 
3-14-1 Hiyoshi, Kohoku-ku, Yokohama, 223-8522, Japan.
}
\email{masahiro.ikeda@riken.jp/masahiro.ikeda@keio.jp}
\address[I. Ishikawa]
{Center for Data Science, 
Ehime University,
3, Bunkyo-cho, Matsuyama, Ehime, 790-8577, Japan/ 
Center for Advanced Intelligence Project, RIKEN
1-4-1, Nihonbashi, Chuo-ku, Tokyo 103-0027, Japan.
}
\email{ishikawa.isao.zx@ehime-u.ac.jp}
\address[K. Taniguchi]
{Advanced Institute for Materials Research,
Tohoku University,
2-1-1 Katahira, Aoba-ku, Sendai, 980-8577, Japan.}
\email{koichi.taniguchi.b7@tohoku.ac.jp}

\subjclass[2020]{Primary 47B33, 30H25; Secondary 46E35;}

\keywords{Composition operators, boundedness, automorphism, Besov spaces, Sobolev spaces, pointwise multipliers}

\begin{document}

\maketitle

\begin{abstract}
This paper aims to characterize boundedness of composition operators on Besov spaces $B^s_{p,q}$ of higher order derivatives $s>1+1/p$ on the one-dimensional Euclidean space.
In contrast to the lower order case $0<s<1$, there were a few results on the boundedness of composition operators for $s>1$.
We prove a relation between the composition operators and pointwise multipliers of Besov spaces, and effectively use the characterizations of the pointwise multipliers. As a result, we obtain necessary and sufficient conditions for the boundedness of composition operators for general $p$, $q$, and $s$ such that $1<p\le \infty$, $0<q\le \infty$, and $s>1+1/p$. 
In this paper, we treat, 
as a map that induces the composition operator, not only a homeomorphism on the real line but also a continuous map whose number of elements of inverse images at any one point is bounded above.
We also show a similar characterization of the boundedness of composition operators on Sobolev spaces.

\end{abstract}

\section{Introduction}\label{sec:1}

In this paper, we characterize a mapping 
$\varphi: \mathbb{R} \to \mathbb{R}$ for which the composition operator 
\[
C_\varphi: f \mapsto f \circ \varphi
\]
is bounded on the inhomogeneous Besov space $B^s_{p,q}=B^s_{p,q}(\mathbb R)$ in one dimension\footnote{We think of the composition operator $C_\varphi$ as the mapping to be defined by $f\mapsto f \circ \varphi$ for continuous functions $f$ as canonical representative if $B^s_{p,q}$ is embedded in the space of continuous functions on $\mathbb R$, i.e., if ``$s>1/p$'' or ``$s=1/p$ and $0<q\le 1$'' (see Proposition \ref{prop:embedding}). Otherwise, the mapping $f\mapsto f \circ \varphi$ is well-defined in the sense of equivalence classes with respect to equality almost evertywhere if $\varphi$ is non-singular (i.e. $|\varphi^{-1}(E)|=0$ for any measurable null set $E \subset \mathbb R$).}. 

The composition operator appears in the context of a change of variables. Its boundedness has long been studied on various function spaces (see, e.g., \cite{Hencl2014,AA2021,Bou-2000,BS-1999,Chae2003,HIIS-2021,Jon-1983,IIS-2022,IIS-appear,Ish-arxiv,KKSS-2014,KXYZ2017,KYZ2011,OP-2017,Sin-1976,Tri_1992,Vod-1989,Vod2005,YYZ-2017} and references therein). In addition, there have been applications to
the analysis of dynamical systems (see \cite{Koo-1931,IIS-appear}), theory of function spaces on domains (see \cite{Tri_1983,Tri_1992,MS_2009}), transport equations (see \cite{Chae2003,FM-arxiv,IIS-appear,Xia-2019}), and so on.
For the composition operators on Besov spaces, Bourdaud and Sickel \cite{BS-1999} provided the necessary and sufficient conditions on homeomorphisms $\varphi$ for the boundedness of $C_\varphi$ on homogeneous and inhomogeneous Besov spaces with low order derivatives $0<s<1$, and later,
Bourdaud \cite{Bou-2000} weakened the assumption on $\varphi$ and gave sufficient conditions on (not necessarily homeomorphisms) $\varphi$ for the boundedness to hold. 
In contrast, the higher order case $s>1$ is also important, but there needs to be more research.
To the best of our knowledge, the characterization was mentioned only for $s \in \mathbb N$ with $s\ge2$ in the framework of Sobolev spaces (see \cite[Remark 1.14]{BS-1999}).
Our purpose is to study the boundedness of $C_\varphi$ on inhomogeneous Besov spaces $B^s_{p,q}$ with $s>1$, and our main results are its necessary and sufficient conditions in the case $s>1+1/p$. 
In this paper, we do not deal with composition operators on homogeneous Besov spaces $\dot B^s_{p,q}$, because they are limited to the trivial ones in the case $s>1$ (see, e.g., \cite[Remark 1.14]{BS-1999} and \cite[Section 6]{Bou-2000}).

To state the results, let us give some notations and definitions. 
We employ the characterization of $B^s_{p,q}$ by differences, i.e, $B^s_{p,q}$ is the collection of all $f \in L^1_{\mathrm{loc}}$ such that  
\[
\|f\|_{B^s_{p,q}}:=
\|f\|_{L^p}
+
\left(
\int_{|h|\le 1}
|h|^{-sq}
\|\Delta_h^mf\|_{L^p}^q
\frac{dh}{|h|}
\right)^\frac{1}{q} < \infty
\]
(with the usual modification for $q=\infty$), 
where $m\in \mathbb N$ with $m>s$.
Here, $L^1_{\mathrm{loc}}=L^1_{\mathrm{loc}}(\mathbb R)$ is the space of all measurable functions on $\mathbb R$ such that $f \in L^1(K)$ for any compact set $K\subset \mathbb R$, and 
the difference operator $\Delta^m_h$ of order $m\in \mathbb N$ is defined by 
\eqref{def:difference_op} below.
See Subsection \ref{sub:2.1} for the details of $B^s_{p,q}$. 
We say that $\varphi$ is Lipschitz if there exists a constant $L>0$ such that 
\[
|\varphi(x) - \varphi(y)|\le L | x-y |,\quad x,y\in \mathbb R.
\]
We denote by $\# A$ the cardinality of a set $A$. For a measurable mapping $\varphi$, we define
\[
U(\varphi) := \sup_{|I| = 1} |\varphi^{-1}(I)|
\quad \text{and}\quad 
\mathcal{M}(\varphi) := \sup_{I} |I|^{-1}|\varphi^{-1}(I)|,
\]
where the supremum is taken over closed bounded intervals in $\mathbb R$. 
We denote by $M(X)$ the set of all pointwise multipliers of a quasi-normed space $X$, i.e., the set of all measurable functions $f$ on $\mathbb R$ such that 
\[
\|f\|_{M(X)}:=
\sup_{\|g\|_{X} \le 1} \|fg\|_{X} < \infty.
\]

The previous results on boundedness of $C_\varphi$ in the case $0<s<1$ are summarized as follows: 

\begin{thm}[Theorems 1.5 and 1.7 in \cite{BS-1999} and Theorems 9 and 11 in \cite{Bou-2000}]\label{thm:previous1}
Let $0<s<1$, $1\le p < \infty$ and $1\le q\le \infty$. Assume that 
\begin{equation}\tag{H}\label{ass:homeo}
    \text{$\varphi : \mathbb R \to \mathbb R$ is a homeomorphism}.
\end{equation}
Then 
$C_\varphi$ is bounded on $B^s_{p,q}$ if and only if 
\begin{equation*}\label{eq.ns-low}
\begin{dcases}
\text{$\varphi^{-1}$ is Lipschitz}\quad &\text{if }0<s<\frac{1}{p},\\
\text{$\varphi$ is Lipschitz and $U(\varphi) < \infty$}\quad &\text{if }\frac{1}{p}< s <1.
\end{dcases}
\end{equation*}
\end{thm}

\begin{thm}[Theorems 7 and 8 in \cite{Bou-2000}]\label{thm:previous2}
Let $0<s<1$, $1\le p < \infty$ and $1\le q\le \infty$. Assume that $\varphi : \mathbb R \to \mathbb R$ is a mapping such that
\begin{equation}\tag{F}\label{ass:finite}
\sup_{x \in \mathbb{R}} \#\varphi^{-1}(x) < \infty.
\end{equation}
Then 
$C_\varphi$ is bounded on $B^s_{p,q}$ if
\[
\begin{dcases}
\text{$\varphi$ is locally absolutely continuous and $\mathcal{M}(\varphi) < \infty$}\quad &\text{if }0<s<\frac{1}{p},\\
\text{$\varphi$ is Lipschitz and $U(\varphi)<\infty$}\quad &\text{if }\frac{1}{p}< s <1.
\end{dcases}
\]
\end{thm}

These theorems are obtained by interpolation from the borderline spaces: Lebesgue spaces $L^p$, H\"older spaces $B^s_{\infty,\infty}$, Sobolev spaces $W^1_p$ and Besov spaces $B^s_{1,1}$ for $1\le p \le \infty$ and $0<s<1$. 
Theorem \ref{thm:previous1} was first proved by \cite{BS-1999}, and then, a {\it priori} condition on $\varphi$ was improved in \cite{Bou-2000}. In addition, \cite{Bou-2000} also provided Theorem \ref{thm:previous2} under the weaker assumption \eqref{ass:finite} than \eqref{ass:homeo}. The results on the borderline spaces are well organized in \cite[Section~3]{Bou-2000}. 

\begin{rem}
Let us give remarks on the above theorems. 
\begin{enumerate}[\rm (a)]
    \item It is easy to see that \eqref{ass:homeo} implies \eqref{ass:finite}, as $\sup_{x \in \mathbb{R}} \#\varphi^{-1}(x)=1$ 
    for any homeomorphism $\varphi$. 

    \item The range of $q$ can be easily extended to $0<q\le \infty$ in Theorems \ref{thm:previous1} and \ref{thm:previous2}.

    \item The critical case $s=1/p$ has been only partially studied (see \cite{Bou-2000,BS-1999,Vod-1989,KKSS-2014,KYZ2011}), but it still has open problems.  
\end{enumerate}
\end{rem}

This paper addresses the higher order case $s>1+1/p$. 
In this case, we may assume that $\varphi$ is continuously differentiable in $\mathbb R$ without loss of generality. In fact, it is necessary if $C_\varphi$ is bounded on $B^s_{p,q}$ with $s>1+1/p$ (see Lemma \ref{lem:nec_holder} below). \\

Our main result is the following:

\begin{thm}\label{thm:main1}
    Let $1 < p < \infty$, $0< q \le \infty$ and $s >1 +1/p$. Assume \eqref{ass:finite}. 
    Then $C_\varphi$ is bounded on $B^s_{p,q}$ if and only if $U(\varphi)<\infty$ and $\varphi' \in M(B^{s-1}_{p,q})$.
\end{thm}

For the case $p=\infty$, we have the following:

\begin{thm}\label{thm:main2}
     Let $0< q \le \infty$ and $s >1$. Then $C_\varphi$ is bounded on $B^s_{\infty,q}$ if and only if $\varphi' \in  M(B^{s-1}_{\infty,q})$.
\end{thm}

Let us here give some remarks on Theorems \ref{thm:main1} and \ref{thm:main2}.
The major difference from the lower order case $0<s<1$ is that the product of functions appears: $C_\varphi f \in B^s_{p,q}$ is roughly $$(C_\varphi f)'=\varphi' \cdot C_\varphi f' \in B^{s-1}_{p,q}$$ in the higher order case $s \ge1$, which implies that the composition operator $C_\varphi$ on $B^{s}_{p,q}$ is related to the pointwise multiplier $\varphi'$ of $B^{s-1}_{p,q}$.
From this, the sufficient condition is derived from a combination of the result on the lower order case $1/p<s<1$ (Theorem~\ref{thm:previous2}) with an inductive argument (see Lemmas \ref{lem:ii-i}, \ref{lem:iii-ii} and \ref{lem:iii-ii_infty}). 
The main part of our proofs 
is the necessity of $\varphi' \in M(B^{s-1}_{p,q})$ (see Lemmas \ref{lem:i-iii} and \ref{lem:i-iii_infty}).
The critical tool is the characterizations of $M(B^{s-1}_{p,q})$. There are many works dealing with pointwise multipliers of Besov spaces and, in particular, 
Nguyen and Sickel \cite{NS-2018} presented the characterizations for $s>1+1/p$:
\[
M(B^{s-1}_{p,q})
= \left\{
f \in L^1_{\mathrm{loc}} \, \bigg|\, 
\sup_{\|\{c_z\}_z\|_{\ell ^p(\mathbb Z)} \le 1}
\Big\|f \sum_{z\in \mathbb Z} c_z \psi(\cdot -z)\Big\|_{B^{s-1}_{p,q}}< \infty
\right\}
\]
for $p\not = \infty$, and $M(B^{s-1}_{\infty,q})=B^{s-1}_{\infty,q}$ (see also Subsection \ref{sub:2.3}). Here, 
$\psi \in C^\infty_0$ is a non-negative function on $\mathbb R$ defined by  \eqref{eq.psi}.
The proofs of the theorems are given in Subsections \ref{sub:3.1} and \ref{sub:3.2}, respectively.

\begin{rem}
If either ``$0< p\le 1$, $0< q \le \infty$ and $s >1 +1/p$" or ``$0<q\le p \le 1$ and $s=1+1/p$", then it is proved that the conditions $U(\varphi)<\infty$ and $\varphi' \in M(B^{s-1}_{p,q})$ are necessary for the boundedness of $C_\varphi$ on $B^s_{p,q}$ under the assumption \eqref{ass:finite} (see Lemma~\ref{lem:i-iii} and Remark \ref{rem:critical}).  It is still open if these are sufficient as long as we know. 
The case $1\le s \le 1+1/p$ remains open.
\end{rem}

We also mention the case of Sobolev spaces $H^s_{p}$. In this case, we also obtain a similar result to Theorem \ref{thm:main1}. This result generalizes the previous result for Sobolev spaces of positive integer order derivatives $s \in \mathbb N$, $s\ge2$, in \cite[Remark 1.14]{BS-1999}.  
The characterization of $M(H^{s-1}_{p})$ with $s>1+1/p$ was presented by Strichartz \cite{Str-1967} (see also \cite{sic-1999}). 
See Section \ref{sec:4} for the details for $H^s_{p}$.\\

We also discuss necessary and sufficient conditions on homeomorphisms $\varphi$ for automorphism of $C_\varphi$ on $B^s_{p,q}$ to hold. 
Here, the automorphism means that $C_\varphi$ is bijective and both $C_\varphi, C_\varphi^{-1}$ ($=C_{\varphi^{-1}}$) are bounded on $B^s_{p,q}$.
As corollaries of Theorems~\ref{thm:main1} and \ref{thm:main2}, 
we immediately have the following:

\begin{cor}\label{cor:main1}
Let $1 < p \le \infty$, $0< q \le \infty$ and $s >1 +1/p$. 
Assume \eqref{ass:homeo}. 
Then $C_\varphi$ is automorphic on $B^s_{p,q}$ if and only if 
$\varphi', (\varphi^{-1})'\in M(B^{s-1}_{p,q})$.
\end{cor}

\begin{rem}
Let us give remarks and related works on the automorphism of $C_\varphi$ on Besov spaces and Sobolev spaces.
\begin{enumerate}[\rm (a)]

\item For $0<s<1$, the characterizations have been studied in \cite{Bou-2000,BS-1999,Vod-1989}.

\item A similar result for the Sobolev spaces to Corollary~\ref{cor:main1} can be also obtained from Theorem \ref{thm:main-TL}.
\end{enumerate}
\end{rem}

\begin{rem}
For all parameters $s\in \mathbb R$ and $0<p,q\le \infty$ (except for $p=\infty$ for the Sobolev spaces), some sufficient conditions for boundedness and automorphism of $C_\varphi$ on $B^s_{p,q}$ and $H^s_{p}$ are known in \cite[Theorem 4.46]{saw_2018} and \cite[Theorem in Subsection 4.3.2]{Tri_1992}, respectively (see also the comments on Theorem~4.46 on page 563 in \cite{saw_2018} for the boundedness). 
Our results improve these for $1<p\le \infty$, $0<q\le \infty$ and $s>1+1/p$.
\end{rem}

\subsection*{Acknowledgement}
This work is based on the discussions at FY2022 IMI Joint Usage Research Program Short-term Joint Research ``Boundedness of Koopman operators on Besov spaces and its applications''.
The first author is supported by Young Scientists Research (No.19K14581), JSPS.
The second author is supported by JST ACTX Grant (No. JPMJAX2004), Japan.
The first and second authors are supported by JST CREST (No. JPMJCR1913), Japan.
The authors thank Professor V. K. Nguyen for giving some useful comments on pointwise multipliers of Besov spaces.

\section{Preliminaries}\label{sec:2}

Let us introduce some notations and definitions used in this paper. For $a, b \ge0$, the symbols $a \lesssim b$ and $b \gtrsim a$ mean that there exists a constant $C>0$ such that $a \le C b$.
The symbol $a \sim b$ means that $a \lesssim b$ and $b \lesssim a$ happen simultaneously. 
We define $\chi_E$ by the characteristic function of $E\subset \mathbb R$. 
We use the notation 
$\|\cdot\|_{X \to Y}$ for the operator norm from a quasi-normed space $X$ to another one $Y$, i.e.,
\[
\|T\|_{X\to Y}
:=
\sup_{\|f\|_{X} = 1} \|Tf\|_{Y}
\]
for an operator $T$ from $X$ into $Y$. For quasi-normed spaces $X$ and $Y$, the notation $X \hookrightarrow Y$ means that $X$ is continuously embedded in $Y$, i.e., $X$ is a subset of $Y$ and there exists a constant $C>0$ such that 
\[
\|f\|_{Y} \le C\|f\|_{X} \quad \text{for any }f\in X.
\]
We denote 
by $BUC=BUC(\mathbb R)$ 
the Banach space of all uniformly continuous and bounded functions on $\mathbb R$ equipped with the supremum norm. 
For $k \in \mathbb N \cup \{\infty\}$, we denote by $C^k=C^k(\mathbb R)$ the space of $k$ times continuously differentiable functions on $\mathbb R$,  
by $C^\infty_0=C^\infty_0 (\mathbb R)$ the space of smooth functions with compact support in $\mathbb R$,  
 by $\mathcal S =\mathcal S(\mathbb R)$ the Schwartz space, which consists of all rapidly decreasing infinitely differentiable functions on $\mathbb R$, and by $\mathcal S' = \mathcal S'(\mathbb R)$ the space of all tempered distributions on $\mathbb R$. The space $\mathcal S'$ is the dual space of $\mathcal S$. We denote by $L^p=L^p(\mathbb R)$ the Lebesgue spaces and 
 by $W^k_p=W^k_p(\mathbb R)$ the Sobolev space for $k\in \mathbb N$ and $0<p\le \infty$.

\subsection{Besov spaces}\label{sub:2.1}
Let $\{\varphi_j\}_{j=0}^\infty$ be the Littlewood-Paley decomposition. More precisely, let $\varphi_0 \in C^\infty_0$ be a non-negative function with $\varphi_0(x)=1$ for $|x|\le 1$ and $\varphi_0(x) =0$ for $|x|\ge2$, and define $\varphi_j$ by 
\[
\varphi_j(x) :=
\varphi_0(2^{-j}x)-\varphi_0(2^{-j+1}x),\quad x\in \mathbb R
\]
for $j\in \mathbb N$. Then $\{\varphi_j\}_{j=0}^\infty$ is the partition of the unity such that 
\[
\sum_{j=0}^\infty \varphi_j(x) = 1,\quad x\in \mathbb R.
\]

\begin{definition}\label{def:fourier}
Let $s \in \mathbb R$ and $0<p,q\le \infty$. Then the Besov space $B^s_{p,q} = B^s_{p,q}(\mathbb R)$ is defined by the collection of all tempered districutions $f \in \mathcal S'$ such that 
\[
\|f\|_{B^s_{p,q,\varphi_0}}:=
\left(
\sum_{j=0}^\infty 
2^{jsq}\left\|\mathscr F^{-1}[\varphi_j \mathscr Ff]\right\|_{L^p}^q
\right)^\frac{1}{q}<\infty
\]
(with the usual modification for $q=\infty$), where $\mathscr F$ and $\mathscr F^{-1}$ are the Fourier transform and its inverse, respectively.
\end{definition}

The Besov space has the characterization by differences. 

\begin{prop}
Let $0 < p,q\le \infty$ and $s > \max \{0, 1/p -1\}$. Then the Besov space $B^s_{p,q}$ is the collection of all functions $f \in L^1_{\mathrm{loc}}$ such that 
\begin{equation}\label{def:Besov-m}
\|f\|_{B^s_{p,q,m}} 
:= 
\|f\|_{L^p}
+
\left(
\int_{|h|\le 1}
|h|^{-sq}
\|\Delta_h^mf\|_{L^p}^q
\frac{dh}{|h|}
\right)^\frac{1}{q} < \infty
\end{equation}
(with the usual modification for $q=\infty$), 
where $m\in \mathbb N$ with $m>s$.
Here, the difference operator $\Delta^m_h$ of order $m\in \mathbb N$ is defined by 
\begin{equation}\label{def:difference_op}
\Delta^m_hf(x) := \sum_{j=0}^m(-1)^{m-j}
\begin{pmatrix}
m\\
j 
\end{pmatrix}
f(x + j h),\quad x, h\in \mathbb R
\end{equation}
and $\Delta^0_h$ is the identity operator.
\end{prop}

Note that $B^s_{p,q}$ is independent of the choice of $\varphi_0$ and $m$.
We write the notation $\|f\|_{B^s_{p,q}}$ as  \eqref{def:Besov-m} and the notation $|f|_{B^s_{p,q}}$ as the second term in \eqref{def:Besov-m}: 
\[
|f|_{B^s_{p,q}}
:=
\left(
\int_{|h|\le 1}
|h|^{-sq}
\|\Delta_h^mf\|_{L^p}^q
\frac{dh}{|h|}
\right)^\frac{1}{q}.
\]
The space $B^s_{p,q}$ is a quasi-Banach space (a Banach space if $p,q\ge 1$).
The following are known.

\begin{prop}\label{prop:embedding}
Let $s\in \mathbb R$ and $0<p,q\le \infty$. Then the following three statements are equivalent: 
\begin{enumerate}[\rm (i)]
\item $B^{s}_{p,q}\hookrightarrow L^\infty$.
\item $B^{s}_{p,q}\hookrightarrow BUC$.
\item There holds either ``$s>1/p$" or ``$s=1/p$ and $q\le 1$".
\end{enumerate}
\end{prop}

\begin{prop}\label{prop:algebra}
Let $s\in \mathbb R$ and $0<p,q\le \infty$. Then the following statements are equivalent: 
\begin{enumerate}[\rm (i)]
\item $B^{s}_{p,q}$ is a multiplication algebra, i.e, $B^{s}_{p,q}\hookrightarrow M(B^{s}_{p,q})$.

\item There holds either ``$s>1/p$" or ``$s=1/p$, $p\not = \infty$ and $q\le 1$".
\end{enumerate}
\end{prop}

For the details of these propositions, we refer to \cite[Lemma 2.2, Remark 2.3, Theorem~3.2, Remark 3.3]{NS-2018} for instance.

\subsection{Some necessary conditions and sufficient conditions}\label{sub:2.2}

In this subsection, we give some preliminary lemmas on necessary conditions and sufficient conditions of boundedness of $C_\varphi$ on $B^s_{p,q}$. 
These lemmas are already known, but we give the proofs for the reader's convenience.

\begin{lem}\label{lem:nec1}
Let $0< p < \infty$, $0< q\le \infty$ and $s > \max\{0,1/p-1\}$. Then, if $C_\varphi$ is bounded on $B^s_{p,q}$, then 
$U(\varphi) < \infty$ and 
\begin{equation}\label{eq.nec1}
U(\varphi)^\frac{1}{p} \lesssim \|C_\varphi\|_{B^s_{p,q} \to B^s_{p,q}}. 
\end{equation}
\end{lem}

\begin{proof}
The proof is the same as in \cite[Subsection 2.3]{BS-1999}. 
Let us take a non-negative function 
$f \in C^\infty_0$ such that $f\equiv1$ on $[0,1]$, and define $f_a(x) := f(x-a)$ for $a\in \mathbb R$. By the assumption that $C_\varphi$ is bounded on $B^s_{p,q}$, there exists a constant $C>0$, independent of $a$, such that 
\[
\begin{split}
\|C_\varphi f_a\|_{L^p} & \le \|C_\varphi f_a\|_{B^s_{p,q}} \le \|C_\varphi\|_{B^s_{p,q} \to B^s_{p,q}} \|f_a\|_{B^s_{p,q}}= \|C_\varphi\|_{B^s_{p,q} \to B^s_{p,q}} \|f\|_{B^s_{p,q}}
\end{split}
\]
for any $a \in \mathbb R$.
Moreover, 
\[
\|C_\varphi f_a\|_{L^p}^p 
\ge 
\int_{\varphi^{-1}([a,a+1])} |f_a(\varphi(x))|^p\, dx = \left|\varphi^{-1}([a,a+1]) \right|
\]
for any $a \in \mathbb R$. 
Combining the above two estimates, we have 
$U(\varphi) < \infty$ and the inequality \eqref{eq.nec1}.
\end{proof}

\begin{lem}\label{lem:suf1}
Let $0<p < \infty$, $0< q \le \infty$ and $s>1/p$. Assume that $U(\varphi)<\infty$. 
Then 
\[
\|C_\varphi f \|_{B^s_{p,q} \to L^p} 
\lesssim U(\varphi)^\frac{1}{p}. 
\]
\end{lem}

\begin{proof}
The proof is the same as in the proof of \cite[Theorem 5]{Bou-2000}. 
It is known that 
\begin{equation}\label{eq.im-Sobolev}
\left(\sum_{j\in \mathbb Z}\|f\|_{L^{\infty}([j,j+1])}^p\right)^\frac{1}{p}
\le  C\|f\|_{B^s_{p,q}}
\end{equation}
for $0< p < \infty$, $0< q \le \infty$ and $s>1/p$. 
The proof of \eqref{eq.im-Sobolev} can be found in \cite[Subsection 3.3]{BS-1999}. 
By using \eqref{eq.im-Sobolev}, we obtain 
\[
\begin{split}
\|C_\varphi f\|_{L^p}^p
& = \sum_{j \in \mathbb Z}
\int_{\varphi^{-1}([j,j+1])} |f(\varphi(x))|^p \, dx\\
& \le U(\varphi) \sum_{j \in \mathbb Z}\|f\|_{L^\infty([j,j+1])}^p \\
&\le C U(\varphi) \|f\|_{B^s_{p,q}}^p
\end{split}
\]
for any $f \in B^s_{p,q}$. 
\end{proof}

\begin{lem}\label{lem:suf2}
Let $0< p < \infty$, $0< q\le \infty$ and $s> \max \{1,1/p\}$. Assume that $C_\varphi$ is bounded on $B^{s-1}_{p,q}$. 
Then
\[
\|C_\varphi\|_{B^s_{p,q} \to L^p}
\lesssim \|C_\varphi\|_{B^{s-1}_{p,q} \to B^{s-1}_{p,q}}.
\]
\end{lem}

\begin{proof}
It follows from a combination of Lemmas \ref{lem:nec1} and \ref{lem:suf1}.
\end{proof}

Next, we shall prove the following result on a necessary condition for the boundedness of $C_\varphi$ on $B^s_{p,q}$. It is one of the fundamental lemmas to prove the necessary condition in Theorem \ref{thm:main1}.

\begin{prop}\label{lem:nec3}
Let $0< p < \infty$ and $0< q\le \infty$.
Assume either ``$s>1+1/p$'' or ``$s>1/p$ and \eqref{ass:homeo}''.  
Then, if $C_\varphi$ is bounded on $B^s_{p,q}$, then $\varphi$ is Lipschitz and 
\begin{equation}\label{eq.nec3}
\|\varphi'\|_{L^\infty}^{s-\frac{1}{p}} \lesssim \|C_\varphi\|_{B^s_{p,q} \to B^s_{p,q}}.
\end{equation}
\end{prop}

To prove this, we prepare the following two lemmas:

\begin{lem}[Proposition 1 in \cite{Bou-2000}]\label{lem:varphi1}
Let $\varphi : \mathbb R \to \mathbb R$ and $A >0$. 
Assume that for any $x_0 \in \mathbb R$, there exists a neighborhood $V$ of $x_0$ such that 
\[
|\varphi(x) - \varphi(y)| \le A|x-y|
\]
for any $x , y\in V$. Then $\varphi$ is Lipschitz and $\|\varphi'\|_{L^\infty}\le A$.
\end{lem}

\begin{lem}\label{lem:nec_holder}
Let $0< p < \infty$, $0< q\le \infty$ and $s>1/p$. Then, if $C_\varphi$ is bounded on $B^s_{p,q}$, then $\varphi$ belongs to $C^{k,\alpha}$ with 
$s - 1/p= k + \alpha$, $k \in \mathbb{N} \cup \{0\}$, $\alpha \in (0,1]$, where $C^{k,\alpha} = C^{k,\alpha}(\mathbb R)$ is the space of all $C^k$-functions such that 
their $k$-th derivatives are H\"older continuous of order $\alpha$.
\end{lem}

\begin{proof}
To begin with, by the Sobolev type embedding, we can see that $\varphi$ is continuous in $\mathbb R$ as $C_\varphi$ is bounded on $B^s_{p,q}$ with $s>1/p$. 
Let $I \subset \mathbb R$ be an open interval with $|I|<\infty$, and take $f \in C^\infty_0$ such that $f(x) = x$ for $x\in I$.
Then $C_\varphi f(x) = \varphi(x)$ for $x \in \varphi^{-1}(I)$, and hence, $\varphi \in B^s_{p,q}(\varphi^{-1}(I))$\footnote{
See e.g. \cite[Section 2]{DS1993} for the definition of Besov spaces $B^s_{p,q}(I)$ on an open set $I \subset \mathbb R$.}. 
Therefore, 
the Sobolev type embedding $B^s_{p,q}(\varphi^{-1}(I)) \subset C^{k,\alpha}(\varphi^{-1}(I))$ for $0< p < \infty$, $0< q\le \infty$ and $s>1/p$ shows that $\varphi \in C^{k,\alpha}(\varphi^{-1}(I))$. 
Since $\mathbb R = \bigcup \varphi^{-1}(I)$, where the union is taken over all open intervals $I \subset \mathbb R$ with $|I|<\infty$, we conclude that $\varphi \in C^{k,\alpha}$.
\end{proof}

\begin{proof}[Proof of Proposition \ref{lem:nec3}]
The proof of the case $1 \le p < \infty$ and $1/p < s < 1$ under the assumption \eqref{ass:homeo} can be found in \cite[Subsection 2.2]{BS-1999}. 
We also use the same idea to prove the other cases.

\smallskip

\noindent{\bf Case: $s>1+1/p$.} 
We give only the proof of the case $1\le s<2$, as the other cases are similarly proved.
By Lemma \ref{lem:nec_holder}, we note that $\varphi \in C^1$. Then, 
for any $x \in \mathbb R$ with $\varphi'(x)\not =0$, there exists an open interval $J \subset \mathbb R$ such that either $\varphi'>0$ on $J$ or $\varphi'<0$ on $J$. Hence, let $b \in \mathbb R$ with $\varphi'(b)\not =0$, and it suffices to prove that 
there exists a constant $C>0$ independent of $b$ such that $|\varphi'(b)|\le C$. 
We give proof only in the case of $\varphi'>0$ on $J$, as the proof of the other case is similar. 
For $\varepsilon >0$, let us take a non-negative function $\eta_\varepsilon \in C^\infty_0$ such that 
\[
\eta'_\varepsilon = \frac{1}{\varepsilon} ( \chi_{[-1-\varepsilon, -1]} -  \chi_{[1, 1+\varepsilon]}).
\]
Then, $\eta_\varepsilon (x) = 1$ for $x\in [-1,1]$, $\supp \eta_\varepsilon \subset [-1-\varepsilon, 1+\varepsilon]$, and 
$|\eta_\varepsilon|_{B^s_{p,q}} \le C \varepsilon^{1/p - s}$, where the constant $C$ is independent of $\varepsilon$.
Take $a,b,c \in J $ such that $a < b < c$ and $b-1 \le c \le a+1$ and 
\[
\varphi(c)-\varphi(b) \le \frac{\varphi(b) - \varphi(a)}{2} \le 1,
\]
and define 
\[
r := \frac{\varphi(b) - \varphi(a)}{2},\quad x_0 := \frac{\varphi(b) + \varphi(a)}{2},
\quad \varepsilon:= 
\frac{\varphi(c)-\varphi(b)}{2}
\]
and 
\[
f(x) := \eta_\varepsilon \left(
\frac{x-x_0}{r}
\right).
\]
Then, 
\[
\|f\|_{L^p} = C\left(\varphi(c) - \varphi(b)\right)^{\frac{1}{p}}
\] 
and 
\[
|f|_{B^s_{p,q}}
=
r^{-\frac{1}{p} + s}|\eta_\varepsilon|_{B^s_{p,q}}
\le C r^{-\frac{1}{p} + s} \varepsilon^{\frac{1}{p} - s}
=
C\left(\varphi(c) - \varphi(b)\right)^{\frac{1}{p} -s}.
\]
By the assumption of boundedness, 
\[
\begin{split}
|C_\varphi f|_{B^s_{p,q}} & \le \|C_\varphi\|_{B^s_{p,q} \to B^s_{p,q}} \|f\|_{B^s_{p,q}} \\
& \le C\|C_\varphi\|_{B^s_{p,q} \to B^s_{p,q}}\left(
\left(\varphi(c) - \varphi(b)\right)^{\frac{1}{p} -s} + \left(\varphi(c) - \varphi(b)\right)^{\frac{1}{p}}
\right).
\end{split}
\]
Noting that 
\[
f(\varphi(x+2h)) - 2 f(\varphi(x+h)) + f(\varphi(x)) = 1
\]
for any $x \in [c-h, b]$ and $h \in [c-b,c-a]$,
we estimate from below
\[
\begin{split}
|C_\varphi f|_{B^s_{p,q}}^q
& \ge 
\int_{c-b}^{c-a}
\left(
\int_{c-h}^b 1\, dx 
\right)^\frac{q}{p} \frac{dh}{h^{1+sq}}\\
& \ge 
\int_{c-b}^{c-a}
(b-c + h)^\frac{q}{p} \frac{dt}{h^{1+sq}}
\ge C (c-b)^{q (\frac{1}{p} -s )}.
\end{split}
\]
Therefore, there exists a constant $C>0$, independent of $b$ and $c$, such that 
\[
\varphi(c) - \varphi(b) \le C(c-b)
\]
for sufficiently close $b$ and $c$. 
Since $\varphi\in C^1$, we obtain $0<\varphi'(b)\le C$ as $c\to b$. 

\smallskip

\noindent{\bf Case: $s>1/p$ and \eqref{ass:homeo}.} 
By the assumption \eqref{ass:homeo}, we may assume that $\varphi$ is strictly increasing in $\mathbb R$ without loss of generality. 
By the same argument as above, for any $b$ and $c$ sufficiently close, there exists a constant $C>0$ independent of $b$ and $c$ such that $\varphi(c) - \varphi(b) \le C(c-b)$. Therefore, the proof is completed by Lemma \ref{lem:varphi1}.
\end{proof}

In the case $p=q=\infty$, we have the following:

\begin{thm}[Theorem 2 in \cite{Bou-2000}]\label{thm:low-infty}
Let $0<s<1$. Then $C_\varphi$ is bounded on $B^s_{\infty,\infty}$ if and only if 
$\varphi$ is Lipschitz.
\end{thm}

As a corollary of Theorem \ref{thm:low-infty}, 
we have the following sufficient condition by the real interpolation argument $(B^{s_0}_{\infty, \infty}, B^{s_1}_{\infty, \infty})_{\theta, q} = B^s_{\infty,q}$ with $s=\theta s_0 + (1-\theta)s_1$, $s_0\not = s_1$ and $\theta \in (0,1)$.

\begin{cor}\label{cor:low-infty}
Let $0<s<1$ and $0<q<\infty$. Then, if 
$\varphi$ is Lipschitz, then $C_\varphi$ is bounded on $B^s_{\infty,q}$.
\end{cor}

\subsection{Characterization of $M(B^{s}_{p,q})$}\label{sub:2.3}

Let $\psi \in C^\infty_0$ be a non-negative function such that 
\begin{equation}\label{eq.psi}
\sum_{z \in \mathbb Z} \psi (\cdot -z) \equiv 1 \text{ on }\mathbb R \quad \text{and}
\quad \supp \psi = [-1,1].
\end{equation}
We define 
\begin{equation}\label{def:B-unif}
B^s_{p,q, \mathrm{unif}}=B^s_{p,q, \mathrm{unif}}(\mathbb R)
:=
\left\{
f \in L^1_{\mathrm{loc}}\, \Big|\,
\sup_{z \in \mathbb Z} \|f \psi(\cdot -z)\|_{B^s_{p,q}} < \infty
\right\}.
\end{equation}
For $p\not =\infty$, we also define
\[
M^s_{p,q}=M^s_{p,q}(\mathbb R)
:=
\left\{
f \in L^1_{\mathrm{loc}} \, \bigg|\, 
\sup_{\|\{c_z\}_z\|_{\ell ^p(\mathbb Z)} \le 1}
\Big\|f \sum_{z\in \mathbb Z} c_z \psi(\cdot -z)\Big\|_{B^{s}_{p,q}}< \infty
\right\}.
\]
It is immediately seen that 
\[
M(B^{s}_{p,q}) \hookrightarrow B^s_{p,q, \mathrm{unif}}.
\]
for $0<p,q\le \infty$ and $s > \max\{0,1/p-1\}$. 
Moreover, we also have the following:

\begin{lem}\label{lem:relation}
Let $0<p<\infty$, $0<q\le \infty$ and $s > \max\{0,1/p-1\}$. Then
\[
M(B^{s}_{p,q}) \hookrightarrow  M^s_{p,q} \hookrightarrow B^s_{p,q, \mathrm{unif}}.
\]
\end{lem}

\begin{proof}
To show the second embedding, we have only to 
take $c_z = \delta_{z,\tilde z}$ for $\tilde z\in \mathbb Z$ (see also the statements before Theorem 1.5 in \cite{NS-2018}).
To see the first embedding, we show that $\sum_{z\in \mathbb Z} c_z \psi(\cdot -z)$ belongs to $B^s_{p,q}$. For this, we write $\psi_z:= \psi(\cdot-z)$ and divide $\mathbb Z$ into $\{\Omega_\ell\}_{\ell=1}^N$ such that 
\[
\dist(\supp \psi_{z_i}, \supp \psi_{z_j}) 
\ge 3m
\]
for any $z_i,z_j \in \Omega_l$ with $z_i \not = z_j$ and for $\ell=1,\ldots, N$. Here, we note that the finite number $N$ depends only on $m$.
Then 
\[
\Big\|
\sum_{z \in \Omega_\ell} c_z \psi_z 
\Big\|_{L^p}^p
=
\sum_{z \in \Omega_\ell} |c_z|^p
\left\| \psi_z 
\right\|_{L^p}^p = \|\{c_z\}_z\|_{\ell^p(\Omega_\ell)}^p \|\psi\|_{L^p}^p,
\]
\[
\Big\|
\Delta_h^m\sum_{z \in \Omega_\ell} c_z \psi_z 
\Big\|_{L^p}^p
=
\sum_{z \in \Omega_\ell} |c_z|^p
\left\| \Delta_h^m\psi_z
\right\|_{L^p}^p
=
\|\{c_z\}_z\|_{\ell^p(\Omega_\ell)}^p \left\| \Delta_h^m\psi
\right\|_{L^p}^p
\]
for $|h|\le 1$ and $\ell=1,\ldots, N$.
Hence, 
\[
\begin{split}
 \Big\|\sum_{z\in \Omega_\ell} c_z \psi_z\Big\|_{ B^s_{p,q}}
&= \Big\|\sum_{z\in \Omega_\ell} c_z \psi_z\Big\|_{L^p}
+
\left(
\int_{|h|\le 1}
|h|^{-sq}
\left(
\Big\|\Delta_h^m \sum_{z\in \Omega_\ell} c_z \psi_z\Big\|_{L^p}^p \right)^\frac{q}{p}
\frac{dh}{|h|}
\right)^\frac{1}{q}\\
& =
 \|\{c_z\}_z\|_{\ell^p(\Omega_\ell)} \|\psi\|_{L^p}\\
& \qquad \qquad +
\left(
\int_{|h|\le 1}
|h|^{-sq}
\left(
\|\{c_z\}_z\|_{\ell^p(\Omega_\ell)}^p \left\| \Delta_h^m\psi
\right\|_{L^p}^p \right)^\frac{q}{p}
\frac{dh}{|h|}
\right)^\frac{1}{q}\\
& = 
 \|\{c_z\}_z\|_{\ell^p(\Omega_\ell)}\|\psi\|_{B^s_{p,q}}
\end{split}
\]
for $\ell=1,\ldots, N$, which implies 
\[
 \Big\|\sum_{z\in \mathbb Z} c_z \psi_z\Big\|_{ B^s_{p,q}}
 \le  \|\{c_z\}_z\|_{\ell^p(\mathbb Z)}\|\psi\|_{B^s_{p,q}}.
\]
The proof of Lemma \ref{lem:relation} is finished.
\end{proof}

Combining Proposition \ref{prop:algebra} and Lemma \ref{lem:relation}, we see that 
\begin{equation}\label{eq.inc1}
B^{s}_{p,q} \hookrightarrow M(B^{s}_{p,q}) \hookrightarrow  M^s_{p,q} \hookrightarrow B^s_{p,q, \mathrm{unif}}
\end{equation}
for ``$0<p<\infty$, $0<q\le \infty$ and $s>1/p$" or `` $0<p<\infty$, $0<q\le 1$ and $s=1/p$", and that 
\begin{equation*}\label{eq.inc2}
B^{s}_{\infty,q} \hookrightarrow M(B^{s}_{\infty,q}) \hookrightarrow B^s_{\infty,q, \mathrm{unif}}
\end{equation*}
for $0<q\le \infty$ and $s>0$. 
The results of \cite{NS-2018} are summarized as follows.

\begin{thm}\label{thm:pm-NS}
Let $0< p,q \le \infty$ and $s>1/p$. Then the following assertions hold:
\begin{enumerate}[\rm (i)]
\item $M(B^{s}_{p,q}) = B^s_{p,q, \mathrm{unif}}$ if and only if
$p \le q$.

\item Let $p\not =\infty$. 
Then $M(B^{s}_{p,q}) = M^s_{p,q}$.

\item 
$M(B^{s}_{p,q}) = B^s_{p,q}$ if and only if $p=\infty$.
\end{enumerate}
\end{thm}

\begin{proof}
For the assertion (i), the sufficiency part was proved by \cite[Theorem 1.2]{NS-2018} and the necessary part was proved by \cite[Corollary 3.18]{NS-2018} (where $p,q \ge1$ is imposed, but the proof shows this assumption can be removed when $s>1/p$). The assertion (ii) was proved by \cite[Theorem 1.5]{NS-2018} when $q<p$ and by a combination of \cite[Theorem 1.2]{NS-2018} with \eqref{eq.inc1} when $p\le q$. For assertion (iii), the sufficiency part was proved by \cite[Theorem 1.7]{NS-2018}. The necessity immediately follows from the fact that 
$1 \in B^{s}_{p,q}$ if and only if $p=\infty$ (see e.g. \cite[Example 2.7, page 241]{saw_2018}) and 
it is clear that $1 \in M(B^{s}_{p,q})$. 
Hence, $p=\infty$ if $M(B^{s}_{p,q}) = B^s_{p,q}$. 
\end{proof}

\begin{rem}
From Proposition \ref{prop:algebra}, Lemma \ref{lem:relation} and Theorem \ref{thm:pm-NS}, we also see the following relations between $B^s_{p,q}$, $B^s_{p,q, \mathrm{unif}}$ and $M^s_{p,q}$ for $s>1/p$:
\begin{enumerate}[\rm (a)]
    \item $B^s_{p,q, \mathrm{unif}} = B^s_{p,q}$ if and only if $p=q=\infty$ (see \cite[Remark 1.8 (iii)]{NS-2018}). 
    \item Let $p\not =\infty$. Then $M^s_{p,q} = B^s_{p,q, \mathrm{unif}}$ if and only if $0< p\le q \le \infty$.
    \item Let $p\not =\infty$. Then $B^s_{p,q} \not =  M^s_{p,q}$. 
\end{enumerate}
\end{rem}

In addition, we have the following result for $s=1/p$. 

\begin{thm}[Theorem 1.9 in \cite{NS-2018}]\label{thm:pm-NS2}
Let $0<p\le 1$ and $s=1/p$. Then the following assertions hold:
\begin{enumerate}[\rm (i)]
\item If $0<p=q\le 1$, then $M(B^{1/p}_{p,p}) = B^{1/p}_{p,p,\mathrm{unif}}$. 

\item If $0<q<p\le 1$, then $M(B^{1/p}_{p,q}) = M^{1/p}_{p,q}$.

\end{enumerate}
\end{thm}

We also have the following embedding.

\begin{prop}[Theorem 2.21 in \cite{Tri_2006}]\label{prop:sobolev}
Let $0< p,q\le \infty$ and $s>\max\{0,1/p-1\}$.
Then 
\[
B^s_{p,q, \mathrm{unif}} \hookrightarrow L^\infty
\]
holds if and only if either ``$s>1/p$" or ``$s=1/p$ and $q\le 1$".
\end{prop}

\section{Proofs of main results}\label{sec:3}

\subsection{Proof of Theorem \ref{thm:main1}}\label{sub:3.1}
We begin by showing the following:

\begin{lem}\label{lem:ii-i}
Let $0< p,q \le \infty$ and $s> \max \{1,1/p\}$. 
Assume that $C_\varphi$ is bounded on $B^{s-1}_{p,q}$ and $\varphi' \in M(B^{s-1}_{p,q})$. 
Then 
$C_\varphi$ is bounded on $B^{s}_{p,q}$.
\end{lem}

\begin{proof}
It follows from the chain rule and Lemma \ref{lem:suf2} that 
\begin{equation}\label{eq.chain}
\begin{split}
\| C_\varphi f\|_{B^s_{p,q}}
& \le 
\|C_\varphi f\|_{L^p}
+
\| \varphi' \cdot C_\varphi f'\|_{B^{s-1}_{p,q}}\\
& \le 
C \left(\|C_\varphi\|_{B^s_{p,q} \to L^p}
+ \|\varphi'\|_{M(B^{s-1}_{p,q})} 
\|C_\varphi\|_{B^{s-1}_{p,q} \to B^{s-1}_{p,q}}
\right)  \|f\|_{B^s_{p,q}}\\
& \le 
C \left(1
+ \|\varphi'\|_{M(B^{s-1}_{p,q})} 
\right)  \|C_\varphi\|_{B^{s-1}_{p,q} \to B^{s-1}_{p,q}} \|f\|_{B^s_{p,q}}
\end{split}
\end{equation}
for any $f \in B^s_{p,q}$.
\end{proof}

Next, we prove the following: 

\begin{lem}\label{lem:i-iii}
Let $0< p < \infty$, $0<q \le \infty$ and $s> 1+1/p$. Assume \eqref{ass:finite}. Then,
if $C_\varphi$ is bounded on $B^{s}_{p,q}$, 
then $U(\varphi)<\infty$ and $\varphi' \in M(B^{s-1}_{p,q})$. 
\end{lem}

For this purpose, we give two auxiliary lemmas. 

\begin{lem}\label{lem: elementary lemma for cond F}
Let $\varphi: \mathbb{R}\to \mathbb{R}$ be a contionuous mapping satisfying \eqref{ass:finite} and $U(\varphi) < \infty$.
Then, for any $a \in \mathbb{R}$ and $b>0$, the closed set $\varphi^{-1}([a-b, a+b])$ is a finite disjoint union of closed intervals $J_1, \dots J_r$.
Moreover, we have
\begin{align}
& \sum_{i=1}^r |J_i| \le 2\lceil b \rceil U(\varphi), \label{ineq: Ia}\\
& r \le \sup_{x \in \mathbb{R}}\# \varphi^{-1}(x),\label{ineq: r}
\end{align}
where $\lceil x \rceil$ is the minimal integer greater than or eqaul to $x$.
\end{lem}
\begin{proof}
Let $J:=[a-b, a+b]$.
We show that $\varphi^{-1}(J)$ is a finite disjoint union of closed intervals.
It follows from the fact that $\partial \varphi^{-1}(J)$ is contained in the finite subset $\varphi^{-1}(\partial J)=\varphi^{-1}(\{a-b, a+b\})$, and thus $\varphi^{-1}(J)$ is necessarily a disjoint union of finite closed intervals.
Let 
\[\varphi^{-1}(J) = \bigsqcup_{j=1}^r J_j.\]
First, we show \eqref{ineq: Ia}.
Let $J_0 \supset J$ be a closed interval of width $2\lceil b \rceil$.
Thus, 
\[ \sum_{j=1}^r |J_i| = |\varphi^{-1}(J)| \le |\varphi^{-1}(J_0)| \le 2\lceil b \rceil U(\varphi).\]
Next, we show \eqref{ineq: r}.
Since $\partial \varphi^{-1}(J) \subset \varphi^{-1}(\partial J)=\varphi^{-1}(\{a-b, a+b\})$, the inequality \eqref{ineq: r} follows from 
\[ 2r =\#\partial \varphi^{-1}(J)  \le \# \varphi^{-1}(\partial J) \le 2 \sup_{x \in \mathbb{R}}\# \varphi^{-1}(x). \]
\end{proof}

\begin{lem}\label{lem: splitting lemma}
Let $\{I_j\}_{j \in \mathbb{N}}$ be a countable sequence of closed intervals of $\mathbb{R}$.
Assume that 
\begin{align}
    \sup_{j \in \mathbb{N}}\#\{ i : I_j \cap I_i \neq \emptyset\} < \infty . \label{cond: finite intersection}
\end{align}
Then, there exists a finite partition $S_1, \dots, S_r$ of $\mathbb{N}$ such that $I_i \cap I_j = \emptyset$ for any $i,j \in S_k$ with $i\neq j$ ($k=1,\dots, r$).
\end{lem}

\begin{proof}
For any nonempty subset $S \subset \mathbb{N}$, we inductively define increasing finite subsets $\tau_1(S) \subset \tau_2(S) \subset \dots \subset S$ as follows:
$\tau_0(\emptyset) :=\emptyset$, $\tau_1(S):= \{\min S\}$, and for $n \ge 1$, 
\begin{align*}
    \tau_{n+1}(S):= 
    \begin{cases}
     \tau_n(S) \cup \left\{ \min\big\{ j \in S \setminus \tau_n(S) : ( \cup_{i \in \tau_n(S)}I_i) \cap I_j = \emptyset\big\}\right\} & (\tau_n(S) \subsetneq S),\\
     \tau_n(S) & (\tau_n(S) = S).
    \end{cases}
\end{align*}
Then, we define $\tau(S):= \cup_{n \ge 0} \tau_n(S)$. 
By the condition \eqref{cond: finite intersection}, for any nonempty subset $S \subset \mathbb{N}$, $\tau(S) \subset S$ is also nonempty and satisfies the following two conditions: 
\begin{enumerate}[\rm (i)]
    \item $I_i \cap I_j =\emptyset$ for any $i,j \in \tau(S)$, and  \label{distinct cond}
    \item $I_j \cap \left( \cup_{i \in \tau(S)} I_i \right) \neq \emptyset$ for any  $j \in S \setminus \tau(S)$. \label{maximal cond}
\end{enumerate}
Then, we inductively construct disjoint sets $S_0,S_1,\dots \subset \mathbb{N}$ as follows:
\begin{align*}
    S_0 &:= \emptyset,\\
    S_{n+1} &:= \tau(\mathbb{N} \setminus \cup_{k=0}^nS_{k})
\end{align*}
for $n \in \mathbb N$.
Let $L :=\sup_{j \in \mathbb{N}}\#\{ i : I_j \cap I_i \neq \emptyset\}$.
We claim that $S_{L+2} = \emptyset$.
In fact, suppose $S_{L+2} \neq \emptyset$.
Fix $j \in S_{L+2}$.
Then, by the construction of $S_k$'s, we have $j \in \mathbb{N} \setminus \cup_{k=0}^{\ell}S_k$ but $j \notin \tau(\mathbb{N} \setminus \cup_{k=0}^{\ell}S_k) = S_{\ell+1}$ for $\ell=0,\dots, L$.
Thus, by \eqref{maximal cond}, for any $\ell=1,\dots, L+1$, there exists $i \in S_\ell$ such that $I_j \cap I_i \neq \emptyset$, 
but it implies $\#\{i : I_j \cap I_i\neq \emptyset\} \ge L+1$ that is contraditions.
Therefore, we conclude that $S_{L+2} = \emptyset$.
Then, Since $\mathbb{N} = \cup_{k=1}^\infty S_k$ and $S_{\ell} = \emptyset$ for all $\ell \ge L+2$, there exists $r \le L+1$ such that $\mathbb{N} = \sqcup_{k=1}^r S_k$.
\end{proof}

\begin{proof}[Proof of Lemma \ref{lem:i-iii}]
We note from Lemmas \ref{lem:nec1} and \ref{lem:nec3} that $\varphi$ is Lipschitz and satisfies $U(\varphi) < \infty$.
Hence, it is enough to show that $\varphi' \in M(B^{s-1}_{p,q})=M^{s-1}_{p,q}$ (see Theorem \ref{thm:pm-NS} (ii)). 

Recall that $\psi \in C^\infty_0$ is a non-negative function satisfying \eqref{eq.psi}, and 
let $\psi_z := \psi(\cdot -z)$ for $z\in \mathbb{Z}$. Define
\[
S_z := \bigcup_{j=1,\dots, m} \bigcup_{|h| \le 1}\supp{(\Delta_{h}^{m-j} \psi_z (\cdot+jh))}.
\]
We take a positive integer $R>0$ such that
\begin{equation}\label{eq.supp2}
2R \ge \|\varphi'\|_{L^\infty} \cdot \max\{{\rm diam}(S_z), 2m\} +2.
\end{equation}
where ${\rm diam}(S):=\sup_{x,y \in S}|x-y|$ for a set $S \subset \mathbb{R}$.
We note that the right-hand side of \eqref{eq.supp2} is independent of the choice of $z \in \mathbb{Z}$.

We divide $\mathbb Z$ into $\{\Omega_\ell\}_{\ell=1}^N$ such that 
\begin{align} \label{disjoint supports}
\dist( S_{z_i}, S_{z_j}) 
\ge 12RU(\varphi)
\end{align}
for any $z_i,z_j \in \Omega_{\ell}$ with $z_i \not = z_j$ and $\ell=1,\ldots, N$. 
Let us fix $\ell \in \{1, \dots, N\}$ for a while.
Since $\varphi$ is Lipschitz and \eqref{eq.supp2} holds, we have
\begin{align*}
    {\rm diam} \left(\varphi(S_z\right)) 
     \le \|\varphi'\|_{L^\infty} {\rm diam}\left(S_z\right)
      \le 2R .
\end{align*}
Thus, for $z \in \Omega_\ell$, there exists $a_z \in \mathbb{R}$ such that
\begin{equation}\label{supp Sz phiinv Ia}
S_z \subset \varphi^{-1}([a_z-R, a_z+R]),
\end{equation}
and in particular,  for any $z \in \Omega_\ell$, we have 
\begin{equation}\label{eq.supp1}
\supp \psi_z \subset \varphi^{-1}([a_z-R, a_z+R]).
\end{equation}
Let $\mathcal{I}_a := [a-2R, a+2R]$.
Then, we claim that
\begin{align} \label{intersection condition for lemma}
     \#\{w \in \Omega_\ell : \mathcal{I}_{a_z} \cap \mathcal{I}_{a_w} \neq \emptyset\} \le \sup_{x \in \mathbb{R}} \#\varphi^{-1}(x) <  \infty
\end{align}
for any $z \in \Omega_\ell$. 
In fact, each connected component (closed interval) of $\varphi^{-1}([a_z -6R, a_z + 6R])$ intersects with at most one $S_z$ for $z\in \Omega_\ell$ by \eqref{ineq: Ia} in Lemma \ref{lem: elementary lemma for cond F} and \eqref{disjoint supports}.
Since $\mathcal{I}_{a_z} \cap \mathcal{I}_{a_w} \neq \emptyset$ implies that $S_z$ intersects with $\varphi^{-1}([a_z - 6R, a_z + 6R])$, the formula \eqref{intersection condition for lemma} follows from \eqref{ineq: r}.
Thus, by Lemma \ref{lem: splitting lemma}, there exists a partition $\Omega_\ell^1,\dots, \Omega_\ell^{r_\ell}$ of $\Omega_\ell$ such that $\mathcal{I}_{a_z} \cap \mathcal{I}_{a_w} =\emptyset$ for $z, w \in \Omega_\ell^k$ for $k=1,\dots,r_\ell$.

We further fix $k \in \{1, \dots, r\}$.
Let us take a function $f \in C^\infty_0$ such that 
$f(x) = x$ for $x \in [-R,R]$ and $\supp f = [-R-1,R+1]$.
For $a \in \mathbb{R}$, we define $f_a(x) := f(x-a)$ and 
\[I_a := \varphi^{-1}([a-R, a+R]).\]
Then,
\begin{equation}\label{eq.key-f}
(C_\varphi f_a(x))' = \varphi'(x)
\end{equation}
for $x \in I_a$.
Moreover, since ${\rm supp}\, (C_\varphi f_a) = \varphi^{-1}([a-R-1, a+R+1])$, we see that for any $z,w \in \Omega_{\ell}^k$ with $z\neq w$
\begin{align} \label{disjoint support for fa}
    \dist({\rm supp}\,(C_\varphi f_{a_z}), {\rm supp}\,(C_\varphi f_{a_w})) \ge \frac{2R-2}{\|\varphi'\|_{L^\infty}} \ge 2m.
\end{align}
Let
\[
\tilde{S}_z := \bigcup_{j=1,\dots, m} \bigcup_{|h| \le 1}\supp{(\Delta_{h}^{j} C_\varphi f_{a_z} )}.
\]
Then, by \eqref{eq.supp2} with \eqref{disjoint support for fa}, for any $z,w \in \Omega_{\ell}^k$ with $z \neq w$, we have
\begin{align} \label{disjoint supports 2}
    \tilde{S}_z \cap \tilde{S}_w = \emptyset
\end{align}
for $z,w \in \Omega_\ell^k$ with $z \neq w$.

Then,  by \eqref{disjoint supports} and \eqref{disjoint supports 2}, we have the following formulas:
\begin{equation}\label{key1}
\Big\|
\sum_{z \in \Omega_\ell^k} c_z \psi_z \varphi'
\Big\|_{L^p}^p
=
\sum_{z \in \Omega_\ell^k} |c_z|^p
\left\| \psi_z \varphi'
\right\|_{L^p}^p,
\end{equation}
\begin{equation}\label{key1-m}
\Big\|
\Delta_h^{m-j}\sum_{z \in \Omega_\ell^k} c_z \psi_z(\cdot + jh) \varphi'
\Big\|_{L^p}^p
=
\sum_{z \in \Omega_\ell^k} |c_z|^p
\left\| \Delta_h^{m-j} \psi_z(\cdot + jh)\varphi'
\right\|_{L^p}^p,
\end{equation}

\begin{equation}\label{key2}
\Big\|
\sum_{z \in \Omega_\ell^k} \left(C_\varphi (|c_z| f_{a_z})\right)'
\Big\|_{L^p}^p
=
\sum_{z \in \Omega_\ell^k} |c_z|^p
\left\| (C_\varphi f_{a_z})'
\right\|_{L^p}^p,
\end{equation}
\begin{equation}\label{key2-m}
\Big\| \Delta_h^j
\sum_{z \in \Omega_\ell^k} \left(C_\varphi (|c_z| f_{a_z})\right)'
\Big\|_{L^p}^p
=
\sum_{z \in \Omega_\ell^k} |c_z|^p
\left\| \Delta_h^j(C_\varphi f_{a_z})'
\right\|_{L^p}^p
\end{equation}
for $|h|\le 1$, $j=1,\ldots, m$, $\ell = 1, \ldots, N$ and $k=1,\ldots, r_\ell$.

By the triangle inequality, we have
\[
\Big\|\varphi' \sum_{z\in \mathbb Z} c_z \psi_z\Big\|_{B^{s-1}_{p,q}}
\le \sum_{\ell =1}^N \sum_{k=1}^{r_\ell}\Big\|\varphi' \sum_{z\in \Omega_\ell^k} c_z \psi_z\Big\|_{B^{s-1}_{p,q}}.
\]
Therefore, the proof of $\varphi' \in M^{s-1}_{p,q}$ is reduced to showing that 
\begin{equation}\label{eq.goal}
\sup_{\|\{c_z\}_z\|_{\ell ^p(\mathbb Z)} \le 1}
\Big\|\varphi' \sum_{z\in \Omega_\ell^k} c_z \psi_z\Big\|_{B^{s-1}_{p,q}}
\le C
\left(
\|\varphi'\|_{L^\infty} + \|C_\varphi\|_{B^s_{p,q} \to B^s_{p,q}} \|f\|_{B^s_{p,q}}
\right)
\end{equation}
for $\ell=1,\ldots, N$ and $k=1,\dots, r_\ell$. 
First, it follows from \eqref{key1}, \eqref{eq.supp1}, \eqref{eq.key-f} and \eqref{key2} that
\begin{equation}\label{eq.first}
\begin{split}
\Big\|\varphi' \sum_{z\in \Omega_\ell^k} c_z \psi_z\Big\|_{L^p}^p
& =
\sum_{z \in \Omega_\ell^k} |c_z|^p
\left\| \psi_z \varphi'
\right\|_{L^p}^p\\
& \le C
\sum_{z \in \Omega_\ell^k} |c_z|^p
\left\| \varphi'
\right\|_{L^p(I_{a_z})}^p\\
& \le C
\sum_{z \in \Omega_\ell^k} |c_z|^p
\left\| (C_\varphi f_{a_z})'
\right\|_{L^p(I_{a_z})}^p\\
& \le C
\Big\|
\sum_{z \in \Omega_\ell^k} \left(C_\varphi (|c_z| f_{a_z})\right)'
\Big\|_{L^p}^p.
\end{split}
\end{equation}
Next, it follows from \eqref{key1-m} that 
\begin{equation}\label{eq.inhom1}
\Big\|\Delta_h^m \varphi' \sum_{z\in \Omega_\ell^k} c_z \psi_z\Big\|_{L^p}^p
=
\sum_{z \in \Omega_\ell^k} |c_z|^p
\left\| \Delta_h^m\psi_z \varphi'
\right\|_{L^p}^p.
\end{equation}
Using the formula 
\[
\Delta_h^m (\psi_z \varphi')(x) 
=
\sum_{j=0}^m 
\begin{pmatrix}
m\\
j
\end{pmatrix}
\Delta_{h}^{m-j} \psi_z (x+jh) \Delta_h^j \varphi'(x),
\]
we have 
\begin{equation}\label{eq.inhom2}
\begin{split}
\|\Delta_h^m\psi_z \varphi'\|_{L^p}
& \le C
\bigg(
\|\Delta_{h}^{m} \psi_z (\cdot+mh)\|_{L^p}\|\varphi'\|_{L^\infty}\\
& \qquad +
\sum_{j=1}^m
\|\Delta_{h}^{m-j} \psi_z (\cdot+jh)\|_{L^\infty}\|\Delta_h^j \varphi'\|_{L^p (\supp (\Delta_{h}^{m-j} \psi_z (\cdot+jh)))}
\bigg).
\end{split}
\end{equation}
Here, we see from \eqref{supp Sz phiinv Ia} and \eqref{eq.key-f} that 
\begin{equation}\label{eq.inhom3}
\|\Delta_h^j \varphi'\|_{L^p (\supp (\Delta_{h}^{m-j} \psi_z (\cdot+jh)))}
\le \|\Delta_h^j \varphi'\|_{L^p (I_{a_z})}
\le \|\Delta_h^j (C_\varphi f_{a_z})'\|_{L^p}.
\end{equation}
Combining \eqref{eq.inhom1}--\eqref{eq.inhom3}, we derive from \eqref{key2-m} that 
\[
\begin{split}
\Big\|\Delta_h^m \varphi' \sum_{z\in \Omega_\ell^k} c_z \psi_z\Big\|_{L^p}^p
& \le C \sum_{z \in \Omega_\ell^k} |c_z|^p
\bigg(
\|\Delta_{h}^{m} \psi_z (\cdot+mh)\|_{L^p}^p
\|\varphi'\|_{L^\infty}^p\\
& \qquad +
\sum_{j=1}^m
\|\Delta_{h}^{m-j} \psi_z (\cdot+jh)\|_{L^\infty}^p
\|\Delta_h^j (C_\varphi f_{a_z})'\|_{L^p}^p
\bigg)\\
& \le 
C 
\bigg(\|\{c_z\}_z\|_{\ell^p}^p \|\varphi'\|_{L^\infty}^p \|\Delta_h^m \psi\|_{L^p}^p\\
& \qquad +
\sum_{j=1}^m
\|\Delta_{h}^{m-j} \psi\|_{L^\infty}^p
\Big\| \Delta_h^j
\sum_{z \in \Omega_\ell^k} \left(C_\varphi (|c_z| f_{a_z})\right)'
\Big\|_{L^p}^p
\bigg).
\end{split}
\]
Hence, 
\begin{equation}\label{eq.second}
\begin{split}
\Big|\varphi' \sum_{z\in \Omega_\ell^k} c_z \psi_z\Big|_{B^{s-1}_{p,q}}
& \le C \|\{c_z\}_z\|_{\ell^p} \|\varphi'\|_{L^\infty}
\|\psi\|_{B^{s-1}_{p,q}}\\
& \qquad +
\sum_{j=1}^m \|\psi\|_{B^{\frac{m-j}{m}(s-1)}_{\infty,\infty}}
\Big\|
\sum_{z \in \Omega_\ell^k} \left(C_\varphi (|c_z| f_{a_z})\right)'
\Big\|_{B^{\frac{j}{m}(s-1)}_{p,q}}.
\end{split}
\end{equation}
Summarizing \eqref{eq.first} and \eqref{eq.second}, we obtain
\[
\begin{split}
\Big\|\varphi' \sum_{z\in \Omega_\ell^k} c_z \psi_z\Big\|_{B^{s-1}_{p,q}}
& \le C
\left(
\|\{c_z\}_z\|_{\ell^p} \|\varphi'\|_{L^\infty} + 
\Big\| \sum_{z \in \Omega_\ell^k} \left(C_\varphi (|c_z| f_{a_z})\right)'
\Big\|_{B^{s-1}_{p,q}}
\right)\\
& \le C
\left(
\|\{c_z\}_z\|_{\ell^p} \|\varphi'\|_{L^\infty} + 
\Big\| \sum_{z \in \Omega_\ell^k} 
C_\varphi (|c_z| f_{a_z})\Big\|_{B^{s}_{p,q}}
\right).
\end{split}
\]

By using 
\[
\Big\|
\sum_{z \in \Omega_\ell^k} |c_z| f_{a_z}
\Big\|_{L^p}^p
\le 
\sum_{z \in \Omega_\ell^k} |c_z|^p
\|  f_{a_z}\|_{L^p}^p
\]
and
\[
\Big\| \Delta_h^m
\sum_{z \in \Omega_\ell^k} |c_z| f_{a_z}
\Big\|_{L^p}^p
\le 
\sum_{z \in \Omega_\ell^k} |c_z|^p
\| \Delta_h^mf_{a_z}
\|_{L^p}^p,
\]
we estimate
\[
\begin{split}
\Big\| \sum_{z \in \Omega_\ell^k} 
|c_z| f_{a_z}\Big\|_{B^{s}_{p,q}}
& \le \left(\sum_{z \in \Omega_\ell^k} |c_z|^p \|f_{a_z}\|_{L^p}^p\right)^\frac{1}{p}\\
& \qquad \qquad +
\left\{
\sum_{k=0}^\infty
2^{ksq} \sup_{|h| < 2^{-k}}
\left(
\sum_{z \in \Omega_\ell^k} |c_z|^p
\| \Delta_h^m f_{a_z}\|_{L^p}^p
\right)^\frac{q}{p}
\right\}^\frac{1}{q}\\
& \le
\|\{c_z\}_z\|_{\ell^p}\|f\|_{L^p}
+
\|\{c_z\}_z\|_{\ell^p} |f|_{B^s_{p,q}}\\
& = \|\{c_z\}_z\|_{\ell^p} \|f\|_{B^s_{p,q}}.
\end{split}
\]

Thus, $\sum_{z \in \Omega_\ell^k} |c_z| f_{a_z}$ is a well-defined element of $B^{s}_{p,q}$.
By the assumption that $C_\varphi$ is bounded on $B^{s}_{p,q}$, we see that 
\[
\begin{split}
\Big\| \sum_{z \in \Omega_\ell^k} 
C_\varphi (|c_z| f_{a_z})\Big\|_{B^{s}_{p,q}}
& =
\left\| C_\varphi \left(\sum_{z \in \Omega_\ell^k} 
|c_z| f_{a_z}\right) \right\|_{B^{s}_{p,q}}\\
& \le \|C_\varphi\|_{B^s_{p,q} \to B^s_{p,q}}
\Big\| \sum_{z \in \Omega_\ell^k} 
|c_z| f_{a_z}\Big\|_{B^{s}_{p,q}}.
\end{split}
\]
By combining the above estimates, we conclude \eqref{eq.goal}. 
The proof of Lemma \ref{lem:i-iii} is finished.
\end{proof}

\begin{rem}\label{rem:ass-finite}
When $p\le q$, the assumption \eqref{ass:finite} in Lemma~\ref{lem:i-iii} can be removed, and the proof becomes simpler. 
In fact, the assumption \eqref{ass:finite} is essentially used to take appropriately the finite partition $\{\Omega_{\ell}^k\}_{\ell,k}$, but 
there is no need to take $\{\Omega_{\ell}^k\}_{\ell,k}$ if we use the characterization $M(B^{s-1}_{p,q}) = B^{s-1}_{p,q, \mathrm{unif}}$ for $p\le q$ from Theorem \ref{thm:pm-NS} {\rm (i)}.
\end{rem}

\begin{rem}\label{rem:critical}
The assertion of Lemma~\ref{lem:i-iii} also holds for $0<q\le p\le 1$ and $s=1+1/p$, since we have the same characterizations of $M(B^{s-1}_{p,q})$ for these parameters from Theorem \ref{thm:pm-NS2}.
\end{rem}

Finally, we show the following:

\begin{lem}\label{lem:iii-ii}
Let $1 < p < \infty$, $0< q \le \infty$ and $s>1+1/p$. 
Assume \eqref{ass:finite}, $U(\varphi^{-1})<\infty$ and $\varphi' \in M(B^{s-1}_{p,q})$. Then 
$C_\varphi$ is bounded on $B^{s-1}_{p,q}$.
\end{lem}

The proof of this lemma is as follows. Combining Proposition \ref{prop:sobolev} with \eqref{eq.inc1}, we see that $\varphi$ is Lipschitz. Hence, when $1+1/p < s <2$, 
the operator $C_\varphi$ is bounded on $B^{s-1}_{p,q}$ by Theorem \ref{thm:previous2}. Thus, the case $1+1/p < s <2$ is proved. 
Next, to prove the case $2\le s \le 2+1/p$, we use the following:

\begin{lem}\label{lem:3.6}
    Let $1<p<\infty$, $0<q\le \infty$ and $1\le s \le 1+1/p$. Assume \eqref{ass:finite}, $U(\varphi^{-1})<\infty$ and $\varphi' \in M(B^{\tilde s}_{p,q})$ for some $\tilde s >1/p$. Then $C_\varphi$ is bounded on $B^{s-1}_{p,q}$.
\end{lem}
\begin{proof}
Lemma \ref{lem:3.6} follows from the interpolation argument between the case $1/p < s < 1$ (Theorem \ref{thm:previous2}) and the case $1+1/p<s<2$ of Lemma \ref{lem:iii-ii}. 
\end{proof}

When $2\le s \le 2+1/p$, it is clear that $\varphi$ satisfies the assumptions of Lemma \ref{lem:3.6}. Hence, $C_\varphi$ is bounded on $B^{s-1}_{p,q}$. The same can be proved inductively for the higher order case $s>2+1/p$. Thus, Lemma \ref{lem:iii-ii} is proved. 

\begin{proof}[Proof of Theorem \ref{thm:main1}]
    Theorem \ref{thm:main1} is immediately proved in a combination of Lemmas \ref{lem:ii-i}, \ref{lem:i-iii}, and \ref{lem:iii-ii}.
\end{proof}

\subsection{Proof of Theorem \ref{thm:main2}}\label{sub:3.2}

Theorem \ref{thm:main2} is proved in a combination of the following two lemmas.

\begin{lem}\label{lem:i-iii_infty}
Let $s>1$ and $0<q \le \infty$. Assume that $C_\varphi$ is bounded on $B^{s}_{\infty,q}$. 
Then $\varphi' \in M(B^{s-1}_{\infty,q})$.
\end{lem}

\begin{proof}[Proof of Lemma \ref{lem:i-iii_infty}]
From Theorem \ref{thm:pm-NS} (iii), it suffices to show that $\varphi' \in B^{s-1}_{\infty,q}$.
Let $s>1$ and $m \in \mathbb N$ with $m>s-1$. 
Let us take $f \in C^\infty_0(\mathbb R)$ such that $f (x) = x$ on $[0,1]$, and define $f_a(x):= f(x-a)$ for $a \in \mathbb R$. Then, since $C_\varphi$ is bounded on $B^{s}_{\infty,q}$, we have
\[
\|(C_\varphi f_a)'\|_{L^\infty}
\le C\|C_\varphi f_a\|_{B^s_{\infty,q}}
\le C \|C_\varphi\|_{B^s_{\infty,q}\to B^s_{\infty,q}} \|f\|_{B^s_{\infty,q}},
\]
where the constant $C$ is independent of $a$. For any $a \in \mathbb R$, we also have 
\[
\|(C_\varphi f_a)'\|_{L^\infty}
=
\|(C_\varphi f'_a) \cdot \varphi'\|_{L^\infty}
\ge \|\varphi'\|_{L^\infty(\varphi^{-1}([a,a+1]))}.
\]
Since $\mathbb R= \bigcup_{a \in \mathbb R}\varphi^{-1}([a,a+1])$, 
we obtain 
\[
\|\varphi'\|_{L^\infty}
\le C \|C_\varphi\|_{B^s_{\infty,q}\to B^s_{\infty,q}}. 
\]

Next, 
let us take $g \in B^s_{\infty,q} \cap C^\infty$ such that 
\[
g'(x) = (-1)^k \quad \text{for }x \in [2(4k-1)m, 2(4k+1)m] 
\]
for $k \in \mathbb Z$. 
Set 
\[
I_m := \bigcup_{k\in \mathbb Z} [2(4k-1)m + m, 2(4k+1)m - m].
\]
Then we have
\[
\sup_{x \in \varphi^{-1}(I_m)} |\Delta_h^m \varphi'(x)| = \sup_{x \in \varphi^{-1}(I_m)} |\Delta_h^m (C_\varphi g)'(x)| 
\]
for $|h|\le 1$. Hence,
\[
\begin{split}
& \int_{|h|\le 1} \sup_{x\in \varphi^{-1}(I_m)}
|\Delta_h^m \varphi'(x)|^q\frac{dh}{|h|^{1+(s-1)q}}\\
& =
\int_{|h|\le 1} \sup_{x\in \varphi^{-1}(I_m)}
 |\Delta_h^m (C_\varphi g)'(x)| ^q\frac{dh}{|h|^{1+sq}}\\
 & = |(C_\varphi g)'|_{B^{s-1}_{\infty,q}}^q 
 \le \|C_\varphi g\|_{B^{s}_{\infty,q}}^q
 \le \|C_\varphi\|_{B^{s}_{\infty,q} \to B^{s}_{\infty,q}}^q  \|g\|_{B^{s}_{\infty,q}}^q.
\end{split}
\]
Similarly, if we translate $g$ by $2m$, $4m$ and $6m$ and make the above argument, we get
\[
 \int_{|h|\le 1} \sup_{x\in \varphi^{-1}(I_m + 2\ell m)}
|\Delta_h^m \varphi'(x)|^q\frac{dh}{|h|^{1+(s-1)q}}
\le \|C_\varphi\|_{B^{s}_{\infty,q} \to B^{s}_{\infty,q}}^q  \|g\|_{B^{s}_{\infty,q}}^q
\]
for $\ell = 1,2,3$. Since 
\[
\mathbb R = \bigcup_{\ell = 0}^3 (I_m + 2\ell m), \quad \text{i.e., } \mathbb R = \bigcup_{\ell = 0}^3 \varphi^{-1}(I_m + 2\ell m), 
\]
we have
\[
\begin{split}
|\varphi'|_{B^{s-1}_{\infty,q}}^q
& =  \int_{|h|\le 1} \sup_{x\in \mathbb R}
|\Delta_h^m \varphi'(x)|^q\frac{dh}{|h|^{1+(s-1)q}}\\
& \le \sum_{\ell =0}^3
 \int_{|h|\le 1} \sup_{x\in \varphi^{-1}(I_m + 2\ell m)}
|\Delta_h^m \varphi'(x)|^q\frac{dh}{|h|^{1+(s-1)q}}\\
& \le  4\|C_\varphi\|_{B^{s}_{\infty,q} \to B^{s}_{\infty,q}}^q  \|g\|_{B^{s}_{\infty,q}}^q.
\end{split}
\]
The proof of Lemma \ref{lem:i-iii_infty} is finished.
\end{proof}

\begin{lem}\label{lem:iii-ii_infty}
Let $s>1$ and $0< q \le \infty$. 
Assume that $\varphi' \in M(B^{s-1}_{\infty,q})$. Then 
$C_\varphi$ is bounded on $B^{s}_{\infty,q}$. 
\end{lem}

\begin{proof}[Proof of Lemma \ref{lem:iii-ii_infty}]
Since $\varphi$ is Lipschitz by Proposition \ref{prop:sobolev}, 
we see that $C_\varphi$ is bounded on $B^{s-1}_{\infty,q}$ with $1<s<2$ from Theorem \ref{thm:low-infty} and Corollary \ref{cor:low-infty}. 
Hence, the proof can be done by a similar inductive argument to Lemma \ref{lem:iii-ii}.
\end{proof}

\section{The case of Sobolev spaces}\label{sec:4}
In this section, we mention the composition operators $C_\varphi$ on the Sobolev spaces $H^s_p$, which are defined by 
\[
H^s_p := \left\{ f \in \mathcal S'  \, \big|\, 
\|f\|_{H^s_p}= \| \mathcal F^{-1}[ (1 + |\xi|^2)^{s/2} \mathcal F f]\|_{L^p}<\infty
\right\}
\]
for $s \in \mathbb R$ and $p > 1$. Similarly to the case of $B^s_{p,q}$, 
we have the following:

\begin{thm}\label{thm:main-TL}
    Let $1 < p < \infty$ and $s >1 +1/p$. Assume \eqref{ass:homeo}. 
    Then $C_\varphi$ is bounded on $H^s_{p}$ if and only if $U(\varphi)<\infty$ and $\varphi' \in M(H^{s-1}_{p})$.
\end{thm}

The proof is done by a similar argument to that of Theorem \ref{thm:main1} with 
the characterization of $H^s_p$ by differences
\begin{equation}\label{eq:S-equiv}
\|f\|_{H^s_p} \sim 
\|f\|_{L^p}
+
\left\|
\left(
\int_0^1
t^{-2s}
\left(
\frac{1}{t}
\int_{|h|\le t}
|\Delta_h^mf(x)|\, dh \right)^2
\frac{dt}{t}
\right)^\frac{1}{2}\right\|_{L^p},
\end{equation}
and 
the following two lemmas. 

\begin{lem}\label{lem:S-low}
Let $1< p < \infty$ and $1/p < s < 1$. Assume \eqref{ass:homeo}. 
Then 
$C_\varphi$ is bounded on $H^s_{p}$ if $\varphi$ is Lipschitz and $U(\varphi)<\infty$. 
\end{lem}

\begin{proof}
The proof is similar to \cite[Subsections 3.2 and 3.3]{BS-1999}.
Similarly to the proof of Lemma \ref{lem:suf1}, it can be shown that 
\[
\|C_\varphi f\|_{L^p} \le CU(\varphi)\|f\|_{H^s_{p}}
\]
for any $f\in H^s_p$ and some constant $C>0$. 
For the second term in \eqref{eq:S-equiv}, we use the complex interpolation
\[
[\mathrm{BMO}, \dot W^1_{p_0}]_{\theta_1} = \dot H^{s_0}_p,\quad s_0 = \theta_1,\quad \frac{1}{p} = \frac{\theta_1}{p_0},\quad 0<\theta_1 <1
\]
for $1<p_0<\infty$, 
and further, the complex interpolation
\[
[\dot H^{s_0}_p, \dot W^1_{p}]_{\theta_2} = \dot H^{s}_p,\quad s = (1-\theta_2) s_0 + \theta_2,\quad 0<\theta_2 <1
\]
(see \cite[Corollary 8.3]{FJ-1990} for the complex interpolations). 
Combining these interpolations with the results on boundedness of $C_\varphi$ on $\mathrm{BMO}$ (\cite[Theorem]{Jon-1983}) and on $\dot W^1_p$ (\cite[Theorem 4]{Bou-2000}), 
and taking $p_0$ close to $1$, 
we see that 
$C_\varphi$ is bounded from $H^s_p$ to the homogeneous Sobolev space $\dot H^{s}_p$ for any $1< p < \infty$ and $1/p < s < 1$ if $\varphi$ is Lipschitz and satisfies \eqref{ass:homeo}.
The proof is finished.
\end{proof}

\begin{lem}\label{lem:multi-TL}
Let $1 < p < \infty$ and $s >1/p$. 
Then 
$M(H^{s}_{p}) = H^s_{p,\mathrm{unif}}$. 
Here, the space $H^s_{p,\mathrm{unif}}$ is similarly defined to \eqref{def:B-unif}.
\end{lem}

Lemma \ref{lem:multi-TL} is a famous result by \cite[Corollary in Section 2 in Chapter II]{Str-1967} (see also \cite[(1.2) in Section 1]{KS-2002} or \cite[Theorem 2.5]{sic-1999}). 

\begin{rem}
    Theorem \ref{thm:main-TL} includes the previous result for the Sobolev spaces $W^k_{p}$ with $k \in \mathbb N$, $k\ge2$ mentioned in \cite[Remark 1.14]{BS-1999}, since $H^k_{p} = W^k_p$ for $k\in \mathbb N$ and $1<p<\infty$.  
\end{rem}

\begin{rem}
We impose the assumption \eqref{ass:homeo} in Theorem \ref{thm:main-TL} and Lemma \ref{lem:S-low}, because the monotonicity of $\varphi$  
is required to use the result on the boundedness on $\mathrm{BMO}$ by \cite[Theorem]{Jon-1983}. 
\end{rem}

\bibliographystyle{plain}   
\bibliography{besov1D}

\begin{bibdiv}
\begin{biblist}

\bib{Bou-2000}{article}{
      author={Bourdaud, G\'{e}rard},
       title={Changes of variable in {Besov} spaces. {II}},
        date={2000},
        ISSN={0933-7741},
     journal={Forum Math.},
      volume={12},
      number={5},
       pages={545\ndash 563},
}

\bib{BS-1999}{article}{
      author={Bourdaud, Gerard},
      author={Sickel, Winfried},
       title={Changes of variable in {Besov} spaces},
        date={1999},
        ISSN={0025-584X},
     journal={Math. Nachr.},
      volume={198},
       pages={19\ndash 39},
}

\bib{Chae2003}{article}{
      author={Chae, Dongho},
       title={On the {Euler} equations in the critical {Triebel-Lizorkin}
  spaces},
        date={2003},
     journal={Arch. Ration. Mech. Anal.},
      volume={170},
      number={3},
       pages={185\ndash 210},
}

\bib{DS1993}{article}{
      author={DeVore, Ronald~A.},
      author={Sharpley, Robert~C.},
       title={Besov spaces on domains in $\mathbb{R}^d$},
        date={1993},
     journal={Trans. Am. Math. Soc.},
      volume={335},
      number={2},
       pages={843\ndash 864},
}

\bib{FM-arxiv}{article}{
      author={Ferreira, Lucas C.~F.},
      author={Machado, Daniel~F.},
       title={On the well-posedness in {Besov-Herz} spaces for the
  inhomogeneous incompressible {Euler} equations},
        date={2023},
     journal={arXiv:2301.07765v1},
}

\bib{FJ-1990}{article}{
      author={Frazier, M.},
      author={Jawerth, B.},
       title={A discrete transform and decompositions of distribution spaces},
        date={1990},
     journal={J. Funct. Anal.},
      volume={93},
}

\bib{HIIS-2021}{article}{
      author={Hatano, Naoya},
      author={Ikeda, Masahiro},
      author={Ishikawa, Isao},
      author={Sawano, Yoshihiro},
       title={Boundedness of composition operators on {Morrey} spaces and weak
  {Morrey} spaces},
        date={2021},
        ISSN={0025-584X},
     journal={J. Inequal. Appl.},
      volume={198},
       pages={19\ndash 39},
}

\bib{Hencl2014}{article}{
      author={Hencl, Stanislav},
      author={Kleprl{\'i}k, Luděk},
      author={Mal{\'y}, Jan},
       title={Composition operator and {Sobolev-Lorentz} spaces {$WL^{n,q}$}},
        date={2014},
     journal={Studia Math.},
      volume={221},
       pages={197\ndash 208},
}

\bib{IIS-2022}{article}{
      author={Ikeda, Masahiro},
      author={Ishikawa, Isao},
      author={Sawano, Yoshihiro},
       title={Boundedness of composition operators on reproducing kernel
  {Hilbert} spaces with analytic positive definite functions},
        date={2022},
        ISSN={0025-584X},
     journal={J. Math. Anal. Appl.},
      volume={511},
       pages={126048},
}

\bib{IIS-appear}{article}{
      author={Ikeda, Masahiro},
      author={Ishikawa, Isao},
      author={Schlosser, Corbinian},
       title={Koopman and {Perron-Frobenius} operators on reproducing kernel
  {Banach} spaces},
        date={2022},
     journal={Chaos},
      volume={32},
       pages={123143},
}

\bib{Ish-arxiv}{article}{
      author={Ishikawa, Isao},
       title={Bounded weighted composition operators on functional
  quasi-{Banach} spaces and stability of dynamical systems},
        date={2023},
     journal={Adv. Math.},
      volume={424},
       pages={109048},
}

\bib{Jon-1983}{article}{
      author={Jones, P.~W.},
       title={Homeomorphisms of the line which preserve {BMO}},
        date={1983},
     journal={Arkiv Math.},
      volume={21},
}

\bib{KS-2002}{article}{
      author={Koch, H.},
      author={Sickel, W.},
       title={Pointwise multipliers of {Besov} spaces of smoothness zero and
  spaces of continuous functions},
        date={2002},
     journal={Rev. Mat. Iberoamericana},
      volume={18},
}

\bib{KKSS-2014}{article}{
      author={Koch, Herbert},
      author={Koskela, Pekka},
      author={Saksman, Eero},
      author={Soto, Tom\'{a}s},
       title={Bounded compositions on scaling invariant {Besov} spaces},
        date={2014},
        ISSN={0022-1236},
     journal={J. Funct. Anal.},
      volume={266},
      number={5},
       pages={2765\ndash 2788},
}

\bib{Koo-1931}{article}{
      author={Koopman, B.~O.},
       title={Hamiltonian systems and transformation in {Hilbert} space},
        date={1931},
        ISSN={0022-1236},
     journal={Proc. Natl. Acad. Sci. USA},
      volume={17},
      number={5},
       pages={315318},
}

\bib{KXYZ2017}{article}{
      author={Koskela, Pekka},
      author={Xiao, Jie},
      author={Zhang, Yi Ru-Ya},
      author={Zhou, Yuan},
       title={A quasiconformal composition problem for the {$ Q $}-spaces},
        date={2017},
     journal={J. Eur. Math. Soc.},
      volume={19},
      number={4},
       pages={1159\ndash 1187},
}

\bib{MS_2009}{book}{
      author={Maz'ya, Vladimir~G.},
      author={Shaposhnikova, Tatyana~O.},
       title={Theory of {Sobolev} multipliers},
      series={Grundlehren der Mathematischen Wissenschaften [Fundamental
  Principles of Mathematical Sciences]},
   publisher={Springer-Verlag, Berlin},
        date={2009},
      volume={337},
        ISBN={978-3-540-69490-8},
        note={With applications to differential and integral operators},
}

\bib{AA2021}{article}{
      author={Menovschikov, Alexander},
      author={Ukhlov, Alexander},
       title={Composition operators on {Hardy-Sobolev} spaces and
  {BMO}-quasiconformal mappings},
        date={2021},
     journal={J. Math. Sci.},
      volume={258},
      number={3},
       pages={313\ndash 325},
}

\bib{NS-2018}{article}{
      author={Nguyen, Van~Kien},
      author={Sickel, Winfried},
       title={On a problem of {Jaak Peetre} concerning pointwise multipliers of
  {Besov} spaces},
        date={2018},
        ISSN={0022-247X},
     journal={Studia Math.},
      volume={243},
      number={2},
       pages={207\ndash 231},
}

\bib{OP-2017}{article}{
      author={Oliva, Marcos},
      author={Prats, Mart\'{\i}},
       title={Sharp bounds for composition with quasiconformal mappings in
  {Sobolev} spaces},
        date={2017},
        ISSN={0022-247X},
     journal={J. Math. Anal. Appl.},
      volume={451},
      number={2},
       pages={1026\ndash 1044},
}

\bib{KYZ2011}{article}{
      author={P.~Koskela, D.~Yang},
      author={Zhou, Y.},
       title={Pointwise characterizations of {Besov} and {Triebel-Lizorkin}
  spaces and quasiconformal mappings},
        date={2011},
     journal={Adv. Math.},
      volume={226},
}

\bib{saw_2018}{book}{
      author={Sawano, Yoshihiro},
       title={Theory of {Besov} spaces},
   publisher={Springer},
        date={2018},
      volume={56},
}

\bib{sic-1999}{incollection}{
      author={Sickel, Winfried},
       title={Pointwise multipliers of {Lizorkin-Triebel} spaces},
        date={1999},
   booktitle={The {Maz'ya} anniversary collection},
   publisher={Springer},
       pages={295\ndash 321},
}

\bib{Sin-1976}{article}{
      author={Singh, R.~K.},
       title={Composition operators induced by rational functions},
        date={1976},
        ISSN={0002-9939},
     journal={Proc. Amer. Math. Soc.},
      volume={59},
      number={2},
       pages={329\ndash 333},
}

\bib{Str-1967}{article}{
      author={Strichartz, R.~S.},
       title={Multipliers on fractional {Sobolev} spaces},
        date={1967},
     journal={J. Math. Mech.},
      volume={16},
}

\bib{Tri_1983}{book}{
      author={Triebel, Hans},
       title={Theory of function spaces},
      series={Monographs in Mathematics},
   publisher={Birkh\"{a}user Verlag, Basel},
        date={1983},
      volume={78},
        ISBN={3-7643-1381-1},
}

\bib{Tri_1992}{book}{
      author={Triebel, Hans},
       title={Theory of function spaces {II}},
      series={Monographs in Mathematics},
   publisher={Birkh\"{a}user Verlag, Basel},
        date={1992},
}

\bib{Tri_2006}{book}{
      author={Triebel, Hans},
       title={Theory of function spaces {III}},
      series={Monographs in Mathematics},
   publisher={Birkh\"{a}user Verlag, Basel},
        date={2006},
      volume={100},
        ISBN={9783764375812},
}

\bib{Vod-1989}{article}{
      author={Vodop{\cprime}yanov, S.~K.},
       title={Mappings of homogeneous groups and embeddings of function
  spaces},
    language={Russian},
        date={1989},
        ISSN={0037-4474},
     journal={Sibirsk. Mat. Zh.},
      volume={30},
      number={5},
       pages={25\ndash 41, 215},
}

\bib{Vod2005}{incollection}{
      author={Vodop{\cprime}yanov, S.~K.},
       title={Composition operators on {Sobolev} spaces},
        date={2005},
   booktitle={{Complex analysis and dynamical systems. II, Proceedings of the
  2nd conference in honor of Professor Lawrence Zalcman's sixtieth birthday
  (Nahariya, Israel 2003), Contemp. Math., vol. 382 (Agranovsky M. et al.,
  eds.), Amer. Math. Soc., Providence, RI; Bar-Ilan University, Ramat Gan}},
   publisher={Providence, RI: American Mathematical Society (AMS); Ramat Gan:
  Bar-Ilan University},
       pages={401\ndash 415},
}

\bib{Xia-2019}{article}{
      author={Xiao, Jie},
       title={The transport equation in the scaling invariant {Besov} or
  {Ess\'{e}n-Janson-Peng-Xiao} space},
        date={2019},
        ISSN={0022-0396},
     journal={J. Differ. Equ.},
      volume={266},
      number={11},
       pages={7124\ndash 7151},
}

\bib{YYZ-2017}{article}{
      author={Yang, Dachun},
      author={Yuan, Wen},
      author={Zhou, Yuan},
       title={Sharp boundedness of quasiconformal composition operators on
  {Triebel-Lizorkin} type spaces},
        date={2017},
        ISSN={1050-6926},
     journal={J. Geom. Anal.},
      volume={27},
      number={2},
       pages={1548\ndash 1588},
}

\end{biblist}
\end{bibdiv}

\end{document}